\documentclass[reqno,12pt]{amsart}
\usepackage{amscd}
\usepackage{amssymb}
\usepackage{amsmath}
\usepackage{amsfonts}
\usepackage{enumerate}
\usepackage{graphicx}
\usepackage{epic}
\usepackage[all,cmtip]{xy}
\usepackage{overpic}
\usepackage{subfig}
\usepackage{caption}
\usepackage{tikz}

\newtheorem{thm}{Theorem}[section]
\newtheorem{prop}[thm]{Proposition}
\newtheorem{lem}[thm]{Lemma}

\newtheorem{conj}[thm]{Conjecture}
\theoremstyle{definition}
\newtheorem{Def}[thm]{Definition}

\newtheorem{rem}[thm]{Remark}
\newtheorem*{rem*}{Remark}

\newcommand{\kk}{\mathbf{k}}
\newcommand{\xx}{\mathbf{x}}

\newcommand{\WW}{\mathcal {W}}

\newcommand{\PP}{\mathcal {P}}

\newcommand{\VV}{\mathcal {V}}

\newcommand{\FF}{\mathcal {F}}

\newcommand{\Zz}{\mathbb {Z}}

\newcommand{\Rr}{\mathbb {R}}
\newcommand{\Qq}{\mathbb {Q}}

\def\w{\widetilde}

\def\xr{\xrightarrow}

\numberwithin{equation}{section}
\usepackage[colorlinks,linktocpage]{hyperref}
\hypersetup{citecolor=blue,linkcolor=blue}
\begin{document}
	\title[Weak Lefschetz property of PL-spheres]{Weak Lefschetz property of PL-spheres}
	\author[F.~Fan]{Feifei Fan}
	\thanks{The author is supported by the National Natural Science Foundation of China (Grant Nos. 11801580, 11871284)}
	\address{Feifei Fan, School of Mathematical Sciences, South China Normal University, Guangzhou, 510631, China.}
	\email{fanfeifei@mail.nankai.edu.cn}
	\maketitle
	\begin{abstract}
A recent result of Papadakis-Petrotou shows that every simplicial sphere has the weak Lefschetz property in characteristic $2$. In this paper, we give a simpler proof of this result for PL-spheres by showing that the weak Lefschetz property in characteristic $2$ is preserved by bistellar moves. 
	Several applications are given.
	\end{abstract}
	\section{Introduction}\label{sec:introduction}
	\renewcommand{\thethm}{\arabic{thm}}
	Our motivating problem is the well-known $g$-conjecture for simplicial spheres.
	It was first proposed by McMullen \cite{Mc71} in 1971 for a complete characterization of the $f$-vectors (i.e. face numbers) of simplicial polytopes. Less than ten years later the $g$-conjecture was proved as a theorem. To describe this theorem, let us review some notions.
	
	For a $(d-1)$-dimensional simplicial complex $\Delta$, the \emph{$f$-vector} of $\Delta$ is
	\[(f_0,f_1,\dots,f_{d-1}),\]
	where $f_i$ is the number of the $i$-dimensional faces of $\Delta$.
	Sometimes it is convenient to set $f_{-1}=1$ corresponding to the empty set. The \emph{$h$-vector} of $\Delta$ is the integer vector
	$(h_0,h_1,\dots,h_d)$ defined from the equation
	\[
	h_0t^d+\cdots+h_{d-1}t+h_d=f_{-1}(t-1)^d+f_0(t-1)^{d-1}+\cdots+f_{d-1}.
	\]
	The $f$-vector and the $h$-vector contain equivalent combinatorial information about $\Delta$, and determine each other by means of linear relations
	\[h_{i}=\sum_{j=0}^{i}(-1)^{i-j}\binom{d-j}{d-i}f_{j-1}, \quad f_{i-1}=\sum_{j=0}^{i}\binom{d-j}{d-i}h_{j}, \quad \text { for } 0 \leq i \leq d.\]
	The sequence $(g_0,g_1,\dots,g_{\lfloor d/2 \rfloor})=(h_0,h_1-h_0,\dots,h_{\lfloor d/2 \rfloor}-h_{\lfloor d/2 \rfloor-1})$ is usually called the \emph{$g$-vector} of $\Delta$.
	
	For any two positive integers $a$ and $i$ there is a unique way to write
	\[a=\binom{a_i}{i}+\binom{a_{i-1}}{i-1}+\cdots+\binom{a_j}{j}\]
	with $a_{i}>a_{i-1}>\cdots>a_{j} \geq j \geq 1$. Define the $i$th \emph{pseudopower of $a$} as
	\[a^{\langle i\rangle}=\binom{a_i+1}{i+1}+\binom{a_{i-1}+1}{i}+\cdots+\binom{a_j+1}{j+1}.\]
	Here $0^{\langle i\rangle}=0$ for all $i$.
	
	\begin{thm}[$g$-theorem]
		An integer vector $(h_0,h_1,\dots, h_d)$ is the $h$-vector of a polytopal $(d-1)$-sphere (the boundary complex of a simplicial $d$-polytope) if and only if
		\begin{enumerate}[(a)]
			\item $h_i=h_{d-i}$ for $i=0,1,\dots,d$ (the Dehn-Sommerville relations);
			\item $1=h_0\leq h_1\leq\cdots\leq h_{\lfloor d/2\rfloor}$;
			\item the $g$-vector satisfies $g_{i+1}\leq g_i^{\langle i\rangle}$ for $i\geq 1$.
		\end{enumerate}
	\end{thm}
	Condition (a) says that the $h$-vector of a simplicial polytope is symmetric. Condition (b) says that the $h$-vector is unimodal, growing up to the middle, while (c) gives a restriction on the rate of this growth. Both (b) and (c) can be reformulated by saying that the $g$-vector
	is an $M$-vector, explaining the name `$g$-theorem'. (A sequence of integers $(k_0,k_1,k_2,\dots)$ satisfies $k_{0}=1$ and  $0\leq k_{i+1} \leq k_{i}^{\langle i\rangle}$ for $i\geq1$ is called an \emph{$M$-sequence}. Finite $M$-sequences are \emph{$M$-vectors}.)
	
	Both necessity and sufficiency parts of the $g$-theorem were proved almost simultaneously (around 1980).
	The sufficiency of McMullen's conditions was proved by Billera and Lee \cite{BL81}, while their necessity was proved by Stanley \cite{S80} by using the theory of toric varieties. In 1993, McMullen gave another proof of the necessity part of $g$-theorem using complicated and difficult convex geometry \cite{Mc93}.
	Later, McMullen simplified his approach in \cite{Mc96}.
	The idea behind both Stanley's and McMullen's proofs was to find a ring whose Hilbert function equals the $h$-vector of the polytope, and satisfying the Hard Lefschetz theorem (Theorem \ref{thm:Lefschetz}).
	
	The $g$-theorem  gives a complete characterisation of integral vectors arising as the $f$-vectors of polytopal spheres.
	It is therefore natural to ask whether the $g$-theorem extends to all sphere triangulations. This question has been regarded as one of the
	main open problems in the theory of face enumeration:
	\begin{conj}[$g$-conjecture]\label{conj:g-con}
		The $g$-theorem holds for all simplicial spheres, or even for all rational homology spheres.
	\end{conj}
	
	In a very recent preprint \cite{A18}, Adiprasito announced a proof
	of this conjecture on its highest level of generality. More recently, Papadakis and Petrotou \cite{PP20} announced another proof of this conjecture by showing that simplicial spheres have the weak Lefschetz property (see Definition \ref{def:WLP}) in characteristic $2$.
In this paper, we focus on the PL-sphere case, and give a simpler proof of  Papadakis-Petrotou result in this case.
The following theorem is the main result of this paper.
	\begin{thm}\label{thm:wlp thm}
		Every PL-sphere has the weak Lefschetz property over any infinite field of characteristic $2$ or $0$.
	\end{thm}

We remark that the statement for characteristic $0$ is implied by the statement for characteristic $2$. This follows from a general result in \cite[Lemma 2.7]{CN2011}.

The proof is elementary and essentially self-contained. Following McMullen's proof of the Hard Lefschetz theorem for simplicial polytopes, the main idea of the proof of Theorem \ref{thm:wlp thm} is to show that the weak Lefschetz property is preserved by bistellar move operations on PL-spheres. For even dimensional spheres, this is known for all bistellar moves except the middle dimensional bistellar move (see the survey article by Swartz \cite{S09}), which is the case we solve in Section \ref{sec:bistellar}.

	Theorem \ref{thm:wlp thm} has several other applications, one of which is to the Gr\"unbaum-Kalai-Sarkaria conjecture. This conjecture is about the embeddability of a simplicial complex $\Delta$ of dimension $d$ into $\Rr^{2d}$. Note that the dimension $2d$ is the highest nontrivial case in a sense that every $d$-complex embeds into $\Rr^{2d+1}$.
	
	It is well known that if a simple planar graph has $f_0$ vertices and $f_1$ edges, then
	$f_1\leq 3f_0$. Actually Euler's combinatorial formula shows that if $f_0\geq3$ this inequality can be strengthened to $f_1\leq 3f_0-6$ (see e.g. \cite[Lemma 1.3.5]{Gun09}).
	Gr\"unbaum \cite{Gr70} asked if there are generalizations of this inequality for $d$-dimensional simplicial complexes that
	allow PL embedding into $\Rr^{2d}$. Kalai and Sarkaria \cite{Ka91,Sar92} suggested a precise formula:
	\begin{conj}[Gr\"unbaum-Kalai-Sarkaria]\label{conj:B-K-S}
		Let $\Delta$ be a simplicial complex of dimension $d$. If there is a PL embedding $\Delta\hookrightarrow \Rr^{2d}$ then
		\[f_d(\Delta)\leq (d+2)f_{d-1}(\Delta).\]
	\end{conj}
	As observed by Kalai-Sarkaria, the number $d+2$ in the inequality is the best possible constant in every dimension.
	They also showed that if every even-dimensional PL-sphere has the weak Lefschetz property, then the conjecture is true.
	So Theorem \ref{thm:wlp thm} implies the validity of Conjecture \ref{conj:B-K-S}  (see Theorem \ref{thm:proof BKS-conj}).	
	\section{Preliminaries}
	\renewcommand{\thethm}{\thesection.\arabic{thm}}
	\subsection{Notations and conventions}\label{subsec:notation}
	Throughout this paper, we assume that $\Delta$  is a  simplicial complex, and we use
	$|\Delta|$ to denote the geometric realization of $\Delta$.
	If $\Delta$ has $m$ vertices, we usually identify the vertices of $\Delta$ with $[m]=\{1,\dots,m\}$.
	We refer to $i$-dimensional faces as \emph{$i$-faces}.
	Let $\FF_i(\Delta)$ be the set of $i$-faces  of $\Delta$.
	A simplicial complex is \emph{pure} if all of its facets (maximal faces) have the same dimension.
	By $\Delta^{m-1}$ we denote the simplicial complex consisting of all subsets of $[m]$; its boundary  $\partial\Delta^{m-1}$ will be the subcomplex of all proper subsets of $[m]$. By abuse of notation, sometimes the symbol $\sigma$ will be used ambiguously to denote a face $\sigma\in\Delta$ and also the simplicial complex consisting of $\sigma$ and all its faces.

	The \emph{link}, \emph{star} and \emph{deletion} of a face $\sigma\in\Delta$ are respectively the subcomplexes
	\[\begin{split}
		\mathrm{lk}_\sigma\Delta=&\{\tau\in\Delta:\tau\cup\sigma\in\Delta,\tau\cap\sigma=\emptyset\};\\
		\mathrm{st}_\sigma\Delta=&\{\tau\in\Delta:\tau\cup\sigma\in\Delta\};\\
		\Delta-\sigma=&\{\tau\in\Delta:\sigma\not\subset\tau\}.
	\end{split}\]
	
	The \emph{join} of two simplicial complexes $\Delta$ and $\Delta'$, where the vertex set $\FF_0(\Delta)$ is disjoint from $\FF_0(\Delta')$ , is the simplicial complex
	\[\Delta*\Delta'=\{\sigma\cup\sigma':\sigma\in\Delta, \sigma'\in\Delta'\}.\]
	In particular, we say that $\Delta^0*\Delta$ is the \emph{cone over $\Delta$} and denoted $\mathrm{cone}\,\Delta$.
	
	Let $\Delta$ be a pure simplicial complex of dimension $d$ and $\sigma\in\Delta$ a $(d-i)$-face such that $\mathrm{lk}_\sigma\Delta=\partial\Delta^i$ and the subset $\tau=\FF_0(\Delta^i)\subset \FF_0(\Delta)$ is not a face of $\Delta$. Then the operation $\chi_\sigma$ on $\Delta$ defined by
	\[\chi_\sigma\Delta=(\Delta-\sigma)\cup(\partial\sigma*\Delta^i)\]
	is called a \emph{bistellar $i$-move}. Obviously we have $\chi_\tau\chi_\sigma\Delta=\Delta$. Two pure simplicial complexes  are \emph{bistellarly equivalent}  if one is transformed to another by a finite sequence of bistellar moves.

	A simplicial complex $\Delta$ is called a \emph{triangulated manifold} (or \emph{simplicial manifold}) if  $|\Delta|$ is a topological manifold. More generally, a $d$-dimensional simplicial complex $\Delta$ is a \emph{$\kk$-homology manifold} ($\kk$ is a field) if
	\[H_*(|\Delta|,|\Delta|-x;\kk)=\w H_*(S^{d};\kk)\quad \text{for all }x\in|\Delta|,\]
	or equivalently, \[H_*(\mathrm{lk}_\sigma\Delta;\kk)=H_*(S^{d-|\sigma|};\kk)\quad \text{for all }\emptyset\neq\sigma\in\Delta.\]
	
	$\Delta$ is a \emph{$\kk$-homology $d$-sphere} if it is a $\kk$-homology $d$-manifold with the same $\kk$-homology as $S^d$.
	(Remark: Usually, the terminology ``homology sphere"  means a manifold having the homology of a sphere. Here we take it in a more relaxed sense than its usual meaning.)
	
	A \emph{piecewise linear (PL for short) $d$-sphere} is a simplicial complex which has a common subdivision with $\partial\Delta^{d+1}$. A \emph{PL $d$-manifold} is a simplicial complex $\Delta$ of dimension $d$ such that $\mathrm{lk}_\sigma\Delta$ is a PL-sphere of dimension  $d-|\sigma|$ for every nonempty face $\sigma\in\Delta$.
	It is obvious that two bistellarly equivalent PL-manifolds are PL homeomorphic, i.e., they have a common subdivision. The following fundamental result shows that the converse is also true.
	\begin{thm}[Pachner {\cite[(5.5)]{P91}}]\label{thm:pachner}
		Two PL-manifolds are bistellarly equivalent if and only if they are PL homeomorphic.
	\end{thm}
	
	As we have seen in Section \ref{sec:introduction}, the $g$-theorem relates the $g$-vectors of \emph{polytopal spheres} to $M$-sequences.
	The importance of $M$-sequence comes from the
	following fundamental result of Macaulay in commutative algebra.
	\begin{thm}[Macaulay {\cite{Mac27}}, {\cite[Theorem I.4.2.10]{BH98}}]\label{thm:m-vector}
		Let $(k_0,k_1,k_2,\dots)$ be a sequence of nonnegative integers. Then  $(k_0,k_1,k_2,\dots)$ is the Hilbert function of a homogeneous quotient of a polynomial ring if and only if $k_0=1$ and $0\leq k_{i+1} \leq k_{i}^{\langle i\rangle}$ for all $i\geq1$.
	\end{thm}
	
	\subsection{Face rings and l.s.o.p}\label{subsec:l.s.o.p.}
	For a commutative ring $\kk$ with unit, let $\kk[x_1,\dots,x_m]$ be the polynomial algebra with one generator for each
	vertex in $\Delta$. It is a graded algebra by setting $\deg x_i=1$.
	The \emph{Stanley-Reisner ideal} of $\Delta$ is
	\[I_\Delta:=(x_{i_1}x_{i_2}\cdots x_{i_k}:\{i_1,i_2,\dots,i_k\}\not\in\Delta)
	\]
	The \emph{Stanley-Reisner ring} (or \emph{face ring}) of $\Delta$ is the quotient \[\kk[\Delta]:=\kk[x_1,\dots,x_m]/I_\Delta.\]
	Since $I_\Delta$ is a monomial ideal, the quotient ring $\kk[\Delta]$ is graded by degree.
	
	For a face $\sigma=\{i_1,\dots,i_k\}\in\FF_{k-1}(\Delta)$, denote by $\xx_\sigma=x_{i_1}\cdots x_{i_k}\in\kk[\Delta]$ the face monomial corresponding to $\sigma$.
	
	Assuming $\kk$ is a field, a set $\Theta=\{\theta_1,\dots,\theta_d\}$ consisting of $d=\dim\Delta+1$ linear forms in $\kk[\Delta]$ is called an \emph{l.s.o.p.} (linear
	system of parameters), if
	\[\kk(\Delta;\Theta):=\kk[\Delta]/(\Theta)\]
	has Krull dimension zero, i.e., it is a finite-dimensional $\kk$-space. We will use the simplified notation $\kk(\Delta)$ for $\kk(\Delta;\Theta)$ whenever it creates no confusion, and write the component
	of degree $i$ of $\kk(\Delta)$ as $\kk(\Delta)_i$. For a subcomplex $\Delta'\subset\Delta$, let $I$ be the ideal of $\kk[\Delta]$ generated by faces in $\Delta\setminus\Delta'$, and denote $I/I\Theta$ by $\kk(\Delta,\Delta';\Theta)$ or simply $\kk(\Delta,\Delta')$.
	
	A linear sequence $\theta_1,\dots,\theta_d$ is an l.s.o.p if and only if the restriction $\Theta_\sigma=r_\sigma(\Theta)$ to each face $\sigma\in\Delta$ generates the polynomial algebra $\kk[x_i:i\in\sigma]$; here $r_\sigma:\kk[\Delta]\to\kk[x_i:i\in\sigma]$ is the projection homomorphism (see \cite[Theorem II.5.1.16]{BH98}). 
	
	It is a general fact that if $\kk$ is an infinite field, then $\kk[\Delta]$ admits an l.s.o.p. (Noether normalization lemma). If $\Theta$ is an l.s.o.p. for $\kk[\Delta]$, then $\kk(\Delta)$ is spanned by the face monomials (see \cite[Lemma III.2.4]{S96}). 
	
	Suppose $\Theta=\{\theta_i=\sum_{j=1}^m a_{ij}x_j\}_{i=1}^d$ is an l.s.o.p. for $\kk[\Delta]$. Then there is an associated $d\times m$ matrix $M_\Theta=(a_{ij})$.
Let $\boldsymbol{\lambda}_i=(a_{1i},a_{2i},\dots,a_{di})^T$ denote the column vector corresponding to the vertex $i\in[m]$. For any ordered subset $S=(i_1,\dots,i_k)\subset [m]$, the submatrix $M_\Theta(S)$ of $M_\Theta$ is defined to be
\[M_\Theta(S)=(\boldsymbol{\lambda}_{i_1},\dots,\boldsymbol{\lambda}_{i_k}).\]
	
	\subsection{Cohen-Macauley and Gorenstein complexes} In this subsection we review some basic combinatorial and algebraic concepts used in the rest of this paper. Throughout this subsection, $\kk$ is an infinite field of arbitrary characteristic.
	
	Let $\Delta$ be a simplicial complex of dimension $d-1$. The face ring $\kk[\Delta]$ is a \emph{Cohen-Macaulay ring} if for any l.s.o.p $\Theta=\{\theta_1,\dots,\theta_d\}$, $\kk(\Delta)$ is a free $\kk[\theta_1,\cdots,\theta_d]$ module. In this case, $\Delta$ is called a \emph{Cohen-Macaulay complex over $\kk$}.
	
	If $\Delta$ is Cohen-Macaulay, the following result of Stanley shows that the $h$-vector of $\Delta$ has a pure algebraic description.
	\begin{thm}[Stanley \cite{Sta75}]\label{thm:stanley}
		Let $\Delta$ be a $(d-1)$-dimensional Cohen-Macaulay complex and let $\Theta=\{\theta_1,\dots,\theta_d\}$ be an l.s.o.p. for $\kk[\Delta]$. Then
		\[\dim_\kk\kk(\Delta)_i=h_i(\Delta),\quad \text{for all } 0\leq i\leq d.\]
	\end{thm}
	
    Let $M$ be a finitely-generated graded $\kk[x_1,\dots,x_m]$-module. The \emph{socle} of $M$ is the following graded submodule of $M$
	\[\mathrm{Soc}\,M:=\{a\in M:x_i\cdot a=0\text{ for all } i\}.\]
	The face ring $\kk[\Delta]$ is a \emph{Gorenstein ring} if $\dim_\kk\mathrm{Soc}\,\kk(\Delta)=1$ for any l.s.o.p $\Theta$, or in other words, $\kk(\Delta)$ is a Poincar\'e duality $\kk$-algebra.
	In this case, $\Delta$ is called a \emph{Gorenstein complex over $\kk$}.
	Further, $\Delta$ is called \emph{Gorenstein*} if $\kk[\Delta]$ is Gorenstein and $\Delta$ is not a cone, i.e., $\Delta\neq\mathrm{cone}\,\Delta'$ for any $\Delta'$.
	
	These algebraic properties of face rings have combinatorial-topological characterisations as follows.
	\begin{thm}\label{thm:algebraic property}
		Let $\Delta$ be a simplicial complex. Then
		\begin{enumerate}[(a)]
			\item {\rm(Reisner \cite{Rei76})} $\Delta$ is Cohen-Macaulay (over $\kk$) if and only if for all faces $\sigma\in\Delta$ (including $\sigma=\emptyset$)
			and $i<\dim\mathrm{lk}_\sigma\Delta$, we have $\w H_i(\mathrm{lk}_\sigma\Delta;\kk)=0$.\vspace{8pt}
			
			\item {\rm(Stanley \cite[Theorem II.5.1]{S96})} $\Delta$ is Gorenstein* (over $\kk$) if and only if it is a $\kk$-homology sphere.\vspace{8pt}
		\end{enumerate}
	\end{thm}
	
	\subsection{Strong and weak Lefschetz property}\label{sec:lefschetz}
	Both Stanley's and McMullen's proof the necessity of the $g$-theorem is by proving the following theorem.
	\begin{thm}[\cite{S80},\cite{Mc93}]\label{thm:Lefschetz}
		If $\Delta$ is the boundary of a simplicial $d$-polytope, then for a certain l.s.o.p. $\Theta$ of $\Qq[\Delta]$, there exists a linear form $\omega$ in $\Qq[\Delta]$ such that the multiplication map
		\[\cdot\omega^{d-2i}:\Qq(\Delta)_{i}\to\Qq(\Delta)_{d-i}\]
		is an isomorphism for all $i\leq d/2$.
	\end{thm}
	
	Let us see why Theorem \ref{thm:Lefschetz} implies the $g$-theorem.
	If $\omega$ satisfies the condition in Theorem \ref{thm:Lefschetz}, then the map $\cdot\omega:\Qq(\Delta)_{i}\to \Qq(\Delta)_{i+1}$ is obviously injective for all $i< d/2 $. This, together with Theorem \ref{thm:stanley}, implies that the quotient ring $\Qq(\Delta)/(\omega)$ is a connected commutative graded $\Qq$-algebra whose $i$th graded component has dimension
	$g_i(\Delta)=h_i(\Delta)-h_{i-1}(\Delta)$ for all $i\leq\lfloor d/2 \rfloor$.
	Hence the $g$-vector of $\Delta$ is an $M$-vector by Theorem \ref{thm:m-vector}.
	
	Let $\Delta$ be a $\kk$-homology sphere. We say $\Delta$ has the \emph{strong Lefschetz property} over $\kk$ if there exists an l.s.o.p. $\Theta$ for $\kk[\Delta]$ and a linear form $\omega$ such that the multiplication map
	\[\cdot\omega^{d-2i}:\kk(\Delta)_{i}\to\kk(\Delta)_{d-i}\]
	is an isomorphism for all $i\leq d/2$.
	\begin{conj}[Algebraic $g$-conjecture]\label{conj:Lefschetz}
	If $\kk$ is an infinite field, then	every $\kk$-homology sphere has the strong Lefschetz property.
	\end{conj}
	Rencently, Adiprasito posted a highly technical preprint \cite{A18} claiming to prove Conjecture \ref{conj:Lefschetz}. More rencently, Papadakis and Petrotou  \cite{PP20} posted another purely algebraic proof of Conjecture \ref{conj:Lefschetz} for the case that $\kk$ has characteristic $2$. The idea in  Papadakis-Petrotou proof is to view the entries $a_{ij}$ of $M_\Theta$ as variables and use a partial differential operator defined on the field of fractions of the polynomial ring
	\[\kk[a_{ij}:1\leq i\leq d,\,1\leq j\leq m].\]
	Our proof of Theorem \ref{thm:wlp thm} is inspired by this idea.

In fact, to prove the $g$-conjecture, it is enough to prove the following weaker property holds.
	\begin{Def}\label{def:WLP}
		Let $\Delta$ be a Cohen-Macaulay complex (over $\kk$). We say that $\Delta$  has the \emph{weak Lefschetz property} (WLP)
		if there is an l.s.o.p. $\Theta$ for $\kk[\Delta]$ and a linear form
		$\omega$ such that the multiplication map $\cdot\omega:\kk(\Delta)_{i}\to\kk
		(\Delta)_{i+1}$ is either
		injective or surjective for all $i$. Such a linear form $\omega$ is called a \emph{weak Lefschetz element} (WLE).
	\end{Def}
	
	\begin{prop}\label{prop:WLP}
		Let $\Delta$ be a $\kk$-homology $(d-1)$-sphere, then $\omega$  is a WLE if the multiplication map
		$\kk(\Delta)_{\lfloor d/2 \rfloor}\xr{\cdot\omega}\kk(\Delta)_{\lfloor d/2 \rfloor+1}$ is surjective.
	\end{prop}
	\begin{proof}
		Suppose the multiplication map $\kk(\Delta)_{\lfloor d/2 \rfloor}\xr{\cdot\omega}\kk(\Delta)_{\lfloor d/2 \rfloor+1}$ is surjective.
		Then $(\kk(\Delta)/(\omega))_{\lfloor d/2 \rfloor+1}=0$. This implies that $(\kk(\Delta)/(\omega))_{i}=0$ for all $i\geq \lfloor d/2 \rfloor+1$,
		and therefore $\kk(\Delta)_{i}\xr{\cdot\omega}\kk(\Delta)_{i+1}$ is surjective for all $i\geq\lfloor d \rfloor$. Since $\Delta$ is Gorenstein,
		$\kk(\Delta)$ is a Poincar\'e duality algebra. Namely, there is an nondegenerate bilinear paring
		\[\kk(\Delta)_{j}\times\kk(\Delta)_{d-j}\to\kk(\Delta)_{d}=\kk,\quad (a,b)\mapsto ab.\]
		It follows that $\kk(\Delta)_{i}\xr{\cdot\omega}\kk(\Delta)_{i+1}$ is injective for $i<\lfloor d/2 \rfloor$.
	\end{proof}
	
	We define a set of pairs $\WW(\Delta)\subset \kk^{f_0}\oplus\kk^{df_0}$ to be
	\[\WW(\Delta)=\{(\omega,\Theta):\Theta\text{ is an l.s.o.p. for }\kk[\Delta]\text{ and }\omega \text{ is a WLE }\}.\]
	It is well known that $\WW(\Delta)$ is a Zariski open set (see e.g. \cite[Proposition 3.6]{Swa06}).
	We will use the term \emph{`generic choice'} of $\Theta$ or $\omega$ to mean that these elements are chosen from a non-empty Zariski open set.

\subsection{A criterion of WLP for even spheres} In this subsection we recall some result about the \emph{canonical module} (see \cite[I.12]{S96} for the definition) of $\kk[\Delta]$ when $\Delta$ is a homology ball, and see how it relates to the WLP of the boundary sphere of $\Delta$ for the case that $\dim\Delta$ is odd.

Let $\Delta$ be a $\kk$-homology $(d-1)$-ball with boundary $\partial \Delta$. Then there is an exact sequence
\begin{equation}\label{eq:exact}
	0\to I\to \kk[\Delta]\to\kk[\partial\Delta]\to 0,
	\end{equation}
where $I$ is the ideal of $\kk[\Delta]$ generated by all faces in $\Delta\setminus\partial\Delta$. By a theorem of Hochster \cite[Theorem  II.7.3]{S96} $I$ is the canonical module of $\kk[\Delta]$. Then from \cite[Theorem I.3.3.4 (d) and Theorem I.3.3.5 (a)]{BH98} we have the following lemma.
\begin{lem}\label{lem:module}
	Let $\Delta$ be a $\kk$-homology $(d-1)$-ball, and $\Theta$ be an l.s.o.p. for  $\kk[\Delta]$. Then there is a nondegenerate bilinear pairing
	\[
		\kk(\Delta)_i\times \kk(\Delta,\partial\Delta)_{d-i}\to \kk(\Delta,\partial\Delta)_d=\kk.
    \]
\end{lem}


\begin{prop}\label{prop:WLP criterion}
	Let $\Delta$ be a $\kk$-homology $(2n-1)$-ball, and $\Theta$ be a generic l.s.o.p for $\kk[\Delta]$. Then the following  are equivalent.
	\begin{enumerate}[(a)]
		\item $\partial\Delta$ has the WLP.
		\item The natural map $\phi:\kk(\Delta,\partial\Delta)_n\to \kk(\Delta)_n$ is an isomorphism.
	\end{enumerate}
\end{prop}
\begin{proof}
Since $\Theta=\{\theta_1,\dots,\theta_{2n}\}$ is generic, we may assume that $\Theta_0=\{\theta_1,\dots,\theta_{2n-1}\}$ is an l.s.o.p. for $\kk[\partial\Delta]$.
 The exact sequence \eqref{eq:exact} induces an exact sequence:
 \[0\to I/(I\cap\Theta)\to \kk(\Delta;\Theta)\to\kk[\partial\Delta]/\Theta\to 0.\]
	Note that $\phi$ factors through the surjection $\kk(\Delta,\partial\Delta):=I/I\Theta\to I/(I\cap\Theta)$, and $\dim_\kk \kk(\Delta,\partial\Delta)_n=\dim_\kk\kk(\Delta,\partial\Delta)_n$ by Lemma \ref{lem:module}.
	Thus $\phi$ is an isomorphism $\Longleftrightarrow$  $(\kk[\partial\Delta]/\Theta)_n=0$ $\Longleftrightarrow$
the multiplication map $\kk(\partial\Delta;\Theta_0)_{n-1}\xr{\cdot\theta_{2n}}\kk(\partial\Delta;\Theta_0)_n$ is an isomorphism  $\Longleftrightarrow$  $(\theta_{2n},\Theta_0)\in\WW(\partial\Delta)$.
\end{proof}

\begin{rem}\label{rem:boundary}
	From the proof of Proposition \ref{prop:WLP criterion} we know that the conditions in Proposition \ref{prop:WLP criterion} only depend on the restriction of $\Theta$ to $\partial\Delta$. 
	In other word, if $\Theta'$ is another l.s.o.p. for $\kk[\Delta]$ such that $M_\Theta(\VV(\partial\Delta))=M_{\Theta'}(\VV(\partial\Delta))$, then the conditions in Proposition \ref{prop:WLP criterion} hold for $\Theta$ if and only if they hold for $\Theta'$. 
\end{rem}

\subsection{The canonical function} In this subsection, we recall a useful result in \cite{PP20}.

\begin{lem}[{\cite[Corollary 4.5]{PP20}}]\label{lem:generator}
	Let $\Delta$ be a $(d-1)$-dimensional $\kk$-homology sphere or ball, $\Theta$ be an l.s.o.p. for $\kk[\Delta]$.  Assume $\sigma_1$ and $\sigma_2$ are two ordered facets of $\Delta$, which have the same orientation in $\Delta$. Then \[\det(M_\Theta(\sigma_1))\xx_{\sigma_1}=\det(M_\Theta(\sigma_2))\xx_{\sigma_2}\] in $\kk(\Delta)_d$ or $\kk(\Delta,\partial\Delta)_d$ respectively. 
\end{lem}

When $\Delta$ is a $(d-1)$-dimensional $\kk$-homology sphere or ball, $\kk(\Delta)_d$ or $\kk(\Delta,\partial\Delta)_d$ is $\kk$, which is spanned by a facet. So each facet $\sigma\in\Delta$ defines a map \[\Psi_\sigma:\kk(\Delta)_d\text{ or } \kk(\Delta,\partial\Delta)_d\to \kk\]
such that for all $f$ in $\kk(\Delta)_d$ or $\kk(\Delta,\partial\Delta)_d$,
\[f=\Psi_\sigma(f)\det(M_\Theta(\sigma))\xx_{\sigma}.\]
Lemma \ref{lem:generator} says that $\Psi_\sigma=\pm\Psi_\tau$ for  any two facets $\sigma,\tau\in\Delta$.
Particullarly, if we fix an orientaion on $\Delta$, this map is independent of the choice of the oriented facet, giving a \emph{canonical function} $\Psi_\Delta:\kk(\Delta)_d\text{ or } \kk(\Delta,\partial\Delta)_d\to \kk$ (see \cite[Remark 4.6]{PP20}).

\section{WLP and bistellar moves}\label{sec:bistellar}
	According to Pachner's theorem (Theorem \ref{thm:pachner}), if we can show that the WLP is preserved by bistellar moves, then Theorem \ref{thm:wlp thm} holds.
	Thanks to Swartz's survey article \cite{Sw14}, we only need to verify this for bistellar moves in a special dimension.
	\begin{thm}[Swartz {\cite[Theorem 3.1 and 3.2]{Sw14}}]\label{thm:bistellar}
		Let $\Delta$ be a $\kk$-homology $(d-1)$-sphere and suppose that $\Delta'$ is obtained from $\Delta$ via a bistellar $i$-move with $i\neq\lfloor d/2\rfloor$. Then
		$\Delta$ has the WLP over $\kk$ if and only if $\Delta'$ dose.
	\end{thm}
	Hence Theorem  \ref{thm:wlp thm} will follow from:
	\begin{thm}\label{thm:WLP for bistellar}
	If $\kk$ has characteristic $2$, then Theorem \ref{thm:bistellar} also holds for $d=2n-1$ and $i=\lfloor d/2\rfloor=n-1$.
	\end{thm}
	Before proving Theorem \ref{thm:WLP for bistellar}, let us see how it implies Theorem  \ref{thm:wlp thm}. The even-dimensional case is obvious.
	For the odd-dimensional case, the link of each vertex of $\Delta$ is an even-dimensional PL-sphere, so we can deduce the WLP of $\Delta$ by applying the following result of Swartz to the links of all vertices of $\Delta$.
	
	\begin{thm}[Swartz {\cite[Theorem 4.26]{S09}}]\label{thm:link wlp}
		Let $\Delta$ be a $(d-1)$-dimensional $\kk$-homology manifold on $[m]$.
		If for at least $m-d$ of the vertices $v$ of $\Delta$ and generic linear form $\omega_v$ and l.s.o.p. $\Theta_v$ of $\kk[\mathrm{lk}_v\Delta]$, the multiplication
		\[\cdot\omega_v:\kk(\mathrm{lk}_v\Delta)_{i-1}\to\kk(\mathrm{lk}_v\Delta)_i\] is surjective, then for generic linear form $\omega$ and l.s.o.p. $\Theta$ of $\kk[\Delta]$, the multiplication $\cdot\omega:\kk(\Delta)_i\to\kk(\Delta)_{i+1}$ is surjective.
	\end{thm}

\subsection{Proof of Theorem \ref{thm:WLP for bistellar}}
Let $\Delta$ be a $\kk$-homology $(2n-2)$-sphere on $[m]$, which has the WLP, $\Delta'=\chi_\sigma\Delta$ with $\dim\sigma=n-1$ and  $\mathrm{lk}_\sigma\Delta=\partial\tau$.
Without loss of generality, we may assume that $\sigma=\{1,\dots,n\}$ and $\tau=\{n+1,\dots,2n\}$.
We need to show that $\Delta'$ also has the WLP. First we will give two lemmas that hold in arbitrary characteristic.
	
In the above notation, let
\[D=\mathrm{cone}\,\Delta=\{v\}*\Delta,\quad D_1=D-\sigma,\quad K=D\cup_{\mathrm{st}_\sigma\Delta} (\sigma*\tau).\]
Then $D$, $D_1$ and $K$ are all $\kk$-homology $(2n-1)$-balls, and $\partial K=\Delta'$.
Let $\Theta=\{\theta_1,\dots,\theta_{2n}\}$ be a generic l.s.o.p for $\kk[K]$. After a linear trasformation if necessary, we may assume that 
\begin{equation}\label{eq:matrix}
	M_\Theta=\bordermatrix{
		& v       &\sigma     &\tau    &U\cr
		& a_1         & \raisebox{-5pt}{{\large $I_n$}}       &    \raisebox{-5pt}{{\large $T$}}   &\raisebox{-5pt}{{\large $*$}}\cr
		& \vdots         &       &     &\cr
		& a_{2n}         &    \raisebox{-0pt}{{\large $0$}}  &\raisebox{0pt}{{\large $I_n$}}     &\raisebox{0pt}{{\large $*$}}\cr
	},
\end{equation}
where $U=\{2n+1,\dots,m\}$, $I_n$ is the $n\times n$ identity matrix, and $T$ is a diagonal matrix:	
\[
T=\begin{pmatrix}
	t_1                &\cdots     & 0 \\
	
	\vdots       &\ddots     &\vdots\\
	0              &\cdots     & t_n
\end{pmatrix}.
\]
To make the subsequent calculation easier, we further require that $a_1,\dots,a_n=0$ and $a_{n+1},\dots,a_{2n}=1$ in \eqref{eq:matrix}. This requirement is reasonable by remark \ref{rem:boundary}.

\begin{lem}\label{lem:basis}
Let $\tau'=\{n+1,\dots,2n-1\}$, and let $\sigma_1=\{v\}\cup\tau'$. Then $\kk(K,\partial K)_n$ has a basis of the form: $\{\xx_\sigma, \xx_{\sigma_1},\dots,\xx_{\sigma_s}\}$, where $\sigma_i\in D_1\setminus\partial D_1$ for $2\leq i\leq s$.
\end{lem}
\begin{proof}
	Let $I_0$ be the ideal of $\kk[\sigma*\tau]$ generated by $\xx_{\sigma}$. Then we have a short exact  sequence:
	\begin{equation}\label{eq:exact1}
		0\to \kk(D,\partial D)_n\to \kk(K,\partial K)_n\to (I_0/I_0\Theta)_n\to 0.
	\end{equation}
	Similarly, let $D_0=\mathrm{st}_\sigma D=\{v\}*\mathrm{st}_\sigma\Delta$, and let $J_0$ be the ideal of $\kk[D_0]$ generated by $x_v$. Then there is another short exact sequence
	\begin{equation}\label{eq:exact2}
		0\to\kk(D_1,\partial D_1)\to \kk(D,\partial D)\to J_0/J_0\Theta\to 0.
	\end{equation}

	Since $D_0=\mathrm{cone}(\sigma*\partial\tau)$, it is easy to see that  $\kk(D_0)_{n-1}=\kk$ is generated by $\xx_{\tau'}$. Note that there is an obvious isomorphism
	\[\kk(D_0)_{*-1}\xr{\cdot x_v} (J_0/J_0\Theta)_*.\]
	Using \eqref{eq:exact2} and this isomorphism, we can pick  a basis $\{\sigma_1,\dots,\sigma_s\}$ for $\kk(D,\partial D)_n$ with $\sigma_i\in D_1\setminus\partial D_1$ for $2\leq i\leq s$.
	\eqref{eq:exact1} shows that $\{\sigma, \sigma_1,\dots,\sigma_s\}$ is a basis for $\kk(K,\partial K)_n$.
\end{proof}	

\begin{lem}\label{lem:kernel}
	If the map $\phi:\kk(K,\partial K)_n\to \kk(K)_n$ is not an isomorphism, then $\ker\phi$ is  spanned by an element of the form:
	\[\alpha=\xx_\sigma+\sum_{i=1}^s l_i\xx_{\sigma_i},\text{ where }l_i\in\kk\ \text{ and }\ l_1=(-1)^n\prod_{i=1}^{n}t_i.\]
\end{lem}
\begin{proof}
	Since $\partial D=\Delta$ has the WLP, Proposition \ref{prop:WLP criterion} shows that the map  $\kk(D,\partial D)_n\to \kk(D)_n$ is an isomorphism. Note that this isomorphism factors through
	\[\kk(D,\partial D)_n\to \kk(K,\partial K)_n\xr{\phi} \kk(K)_n.\]
	Hence if $\phi$ is not an isomorphism, then $\dim_\kk\ker\phi=1$, and by Lemma \ref{lem:basis}, $\ker\phi$ is spanned by an element $\alpha$ of the form in the lemma. 
	
	Since $\mathrm{st}_\sigma K=\sigma*\partial(\tau\cup\{v\})$, the expression of $M_\Theta$ implies that $\xx_\sigma$ restricts to $(-1)^{n-1}(\prod_{i=1}^{n}t_i)\xx_{\sigma_1}$ in $\kk(\mathrm{st}_\sigma K)$. Thus, the equation \[\phi_\sigma\circ\phi(\alpha)=\xx_\sigma+l_1\xx_{\sigma_1}=0,\ \text{ where }\ \phi_\sigma:\kk(K)\to \kk(\mathrm{st}_\sigma K),\]
	gives $l_1=(-1)^{n}\prod_{i=1}^{n}t_i$.
\end{proof}

We are now ready to prove Theorem \ref{thm:WLP for bistellar}.

\begin{proof}[Proof of Theorem \ref{thm:WLP for bistellar}]
By proposition \ref{prop:WLP criterion}, we only need to show that the map $\phi:\kk(K,\partial K)_n\to \kk(K)_n$ is an isomorphism, since $\partial K=\Delta'$. 

Set $\beta=\sum_{i=1}^s l_i\xx_{\sigma_i}$. If $\phi$ is not an isomorphism, then Lemma \ref{lem:kernel} implies that $\xx_\sigma=\beta$ in $\kk(D)$. Note that $\mathrm{char}\,\kk=2$. 
Hence we have
\begin{equation}\label{eq:square}
\Psi_D(\beta^2)=\Psi_D(\xx_\sigma\beta)=l_1\Psi_D(\xx_\sigma\xx_{\sigma_1})=t_1\cdots t_n,
\end{equation}
where the second equality follows by the fact that $\sigma_i\not\in\mathrm{st}_\sigma\Delta$ for $i\geq 2$ by Lemma \ref{lem:basis}, and the third equality uses Lemma \ref{lem:kernel} and the fact that $\det(M_\Theta(\sigma_1\cup\sigma))=1$.
 
Viewing $t_i$ as variebles, the canonical function $\Psi_D$ and the above coefficients $l_i$ are then functions on $t_1,\dots,t_n$. 
Define a partial diffenrential operator $\PP:=\frac{\partial^n}{\partial t_1\cdots\partial t_n}$.
Then by \eqref{eq:square}, we have
\[\PP\Psi_D(\beta^2)=\PP\sum_{i=1}^s l_i^2\Psi_D(\xx_{\sigma_i}^2)=\sum_{i=1}^s l_i^2\PP\Psi_D(\xx_{\sigma_i}^2)=1,\]
where the first and second equality both come from the assumption that $\mathrm{char}\,\kk=2$.
However, since $\tau\not\in D$, it follows that for any $\sigma_i$, there is some $n_i\in n+1,\dots,2n$ with $n_i\not\in\FF_0(L_i)$, where $L_i=\mathrm{st}_{\sigma_i}D$. This implies that $t_{n_i}$ dose not appear in the function $\Psi_D(\xx_{\sigma_i}^2)=\Psi_{L_i}(\xx_{\sigma_i}^2)$. Therefore $\PP\Psi_D(\xx_{\sigma_i}^2)=0$ for all $1\leq i\leq s$, a contradiction.
\end{proof}

\section{Applications of Theorem \ref{thm:wlp thm}}\label{sec:application}
\subsection{Enumeration of bistellar moves}
$g$-theorem for PL-spheres can be applied to the enumeration problem of bistellar moves on a PL-sphere.
	\begin{thm}
		Let $\Delta$ be a PL-sphere of dimension $d$. Then the number of $k$-moves in the sequence of bistellar moves taking $\Delta$ to $\partial\Delta^{d+1}$ can not exceed  the number of $(d-k)$-moves for $k\leq\lfloor d/2\rfloor$.
	\end{thm}
	This is a direct consequence of the non-negativity of the $g$-vector of a PL-sphere and the following elementary result that describe the effect of a bistellar move on the h-vector.
	Note that the $g$-vector of $\partial\Delta^{d+1}$ is $(1,0,\dots,0)$.
	\begin{prop}[see i.e. {\cite[Proposition 5.1]{S09}}]
		If a triangulated $d$-manifold $\Delta'$ is obtained from $\Delta$ by a bistellar $k$-move, $0\leq k\leq\lfloor\frac{d-1}{2}\rfloor$, then
		\begin{align*}
			g_{k+1}(\Delta')&=g_{k+1}(\Delta)+1;\\
			g_{i}(\Delta')&=g_{i}(\Delta)\ \text{ for } i\neq k+1.
		\end{align*}
		Furthermore, if $d$ is even and $\Delta'$ is obtained from $\Delta$ by a bistellar $d/2$-move, then
		\[g_{i}(\Delta')=g_{i}(\Delta)\ \text{ for all } i.\]
	\end{prop}

\subsection{WLP of simplicial toric veriaties}
	Here is another application of Theorem  \ref{thm:wlp thm} to toric algebraic geometry. Consider $\Sigma$ a complete simplicial fan in $\Rr^d$. Here we require $\Sigma$ to be rational, i.e., each ray  of $\Sigma$ is generated by a primitive vector $\boldsymbol{\lambda}_i=(\lambda_{1i},\dots,\lambda_{di})\in \Zz^d$.
Let $\Delta_\Sigma$ be the \emph{underlying simplicial complex} of $\Sigma$. By definition, $\{i_1,\dots,i_k\}\subset[m]$ is a face of $\Delta_\Sigma$ if and
only if $\boldsymbol{\lambda}_{i_1},\dots,\boldsymbol{\lambda}_{i_k}$ span a cone of $\Sigma$.
It is easy to see that the vectors $\boldsymbol{\lambda}_1,\dots,\boldsymbol{\lambda}_m$ define an l.s.o.p. for $\Qq[\Delta_\Sigma]$: \[\Theta_\Sigma=\{\theta_i=\lambda_{i1}x_1+\cdots+\lambda_{im}x_m\}_{i=1}^d.\]
The rational cohomology of the associated toric variety $X_\Sigma$ can be calculated as follows:
\begin{thm}[{\cite[Danilov]{Dan78}}, see also {\cite[Theorem 12.4.1]{CLS11}}]\label{thm:toric variety}
	Let $\Sigma$ be a complete simplicial fan in $\Rr^d$. Then there is a ring isomorphism
	\begin{gather*}
		H^*(X_\Sigma;\Qq)\cong \Qq(\Delta_\Sigma;\Theta_\Sigma),\\
		H^{2i}(X_\Sigma;\Qq)\cong \Qq(\Delta_\Sigma;\Theta_\Sigma)_i,\quad H^{2i+1}(X_\Sigma;\Qq)=0,
	\end{gather*}
\end{thm}
Note that when $\Sigma$ is complete and simplcial, $\Delta_\Sigma$ is a PL $(d-1)$-sphere, so from Theorem  \ref{thm:wlp thm} and Theorem \ref{thm:toric variety} we get the following theorem, which is a `weak' generalization of the hard Lefschetz theorem on projective toric varieries to all simplicial toric varieties.
\begin{thm}
	Let $\Sigma$ be a complete simplicial fan in $\Rr^d$, $X_\Sigma$ the associated toric variety, $\mu_{i}=\dim H^{2i}(X_\Sigma;\Qq)$. Then the vector
	\[\{\mu_0,\mu_1-\mu_0,\dots,\mu_{\lfloor d/2 \rfloor}-\mu_{\lfloor d/2 \rfloor-1}\}\]
	is a $M$-vector. Moreover, if the rays of $\Sigma$ are in generic positions, then $H^*(X_\Sigma;\Qq)$ has the WLP.
\end{thm}

	\subsection{GKS conjecture}
	While it is well known that the Gr\"unbaum-Kalai-Sarkaria conjecture is a consequence of Theorem  \ref{thm:wlp thm}, we include its proof here by following \cite[Corollary 4.8]{A18}, for the sake of completeness.
	
	\begin{thm}\label{thm:proof BKS-conj}
	Conjecture \ref{conj:B-K-S} holds.
	\end{thm}
	\begin{proof}
	Suppose  $\Delta$ can be piecewise linearly embedded into $\Rr^{2d}$. Then there exists a PL $2d$-sphere $K$ containing $\Delta$ as a subcomplex. By Theorem  \ref{thm:wlp thm}, there exists $(\omega,\Theta)\in\WW(K)$ for any infinite field $\kk$ of characteristic $2$. Hence in the following commutative diagram, the top horizontal map  is a surjection
	\[\xymatrix{
		\kk(K)_d \ar[r]^-{\cdot\omega} \ar[d] & \kk(K)_{d+1}\ar[d]\\
		\kk(K)_d\ar[r]^-{\cdot\omega} &\kk(\Delta)_{d+1}}
	\]
   Since the vertical projections are clearly surjective, the bottom horizontal map is also surjective. So we have
   \begin{equation}\label{eq:inequality 1}
   	\dim_\kk\kk(K)_{d+1}\leq\dim_\kk\kk(K)_d\leq f_{d-1}(\Delta).
   \end{equation}
   The second inequality comes from the fact that $\kk(\Delta)$ is spanned by face monomials.

   Furthermore, for any $(d-1)$-face $\sigma\in\Delta$, choose a facet $\tau\in K$ such that $\sigma\subset\tau$. Without loss of generility, we may assume $\sigma=\{1,\dots,d\}$, $\tau=\{1,\dots,2d+1\}$ and
   \[M_\Theta=(I_{2d+1}\mid A).\]
   Hence the generators $x_1,\dots,x_d$ do not appear in the expressions of the $d+1$ linear forms $\theta_{d+1}\dots,\theta_{2d+1}$. Multiplying $\xx_\sigma$ by $\theta_i$ for each $i\geq d+1$ produces exactly $d+1$ linearly independent relations among the face monomials in $\kk(K)_{d+1}$.
   So there are at most $(d+1)f_{d-1}(\Delta)$ linearly independent relations between the generators, i.e. $d$-face monomials, of $\kk(K)_{d+1}$. It follows that
  \begin{equation}\label{eq:inequality 2}
  	\dim_\kk\kk(K)_{d+1}\geq f_d(\Delta)-(d+1)f_{d-1}(\Delta).
 \end{equation}	
Combining \eqref{eq:inequality 1} with \eqref{eq:inequality 2} gives the desired inequality.
	\end{proof}	
	
	\bibliography{M-A}
	\bibliographystyle{amsplain}
\end{document}